\documentclass{amsart}

\usepackage{amsmath,amssymb,enumerate}

\newcommand{\dbarb}{\overline{\partial}_b}

\newcommand{\zbar}{\overline{z}}
\newcommand{\Zbar}{\overline{Z}}

\newcommand{\omegabar}{\overline{\omega}}

\newcommand{\hZbar}{\hat{\Zbar}}
\newcommand{\hu}{\hat{u}}

\newcommand{\hf}{\hat{f}}
\newcommand{\htheta}{\hat{\theta}}
\newcommand{\hdbarb}{\hat{\dbarb}}
\newcommand{\homegabar}{\hat{\overline{\omega}}}

\newtheorem{thm}{Theorem}
\newtheorem{prop}[thm]{Proposition}
\newtheorem{lemma}[thm]{Lemma}
\newtheorem{cor}[thm]{Corollary}

\begin{document}

\title[The tangential Cauchy-Riemann complex]{The tangential Cauchy-Riemann complex on the Heisenberg group Via Conformal Invariance}
\author{Chin-Yu Hsiao}
\address{Universit{\"a}t zu K{\"o}ln,  Mathematisches Institut, Weyertal 86-90, 50931 K{\"o}ln, Germany}
\email{chsiao@math.uni-koeln.de}

\author{Po-Lam Yung}
\address{Department of Mathematics, Rutgers, the State University of New Jersey, Piscataway, NJ 08854}
\email{pyung@math.rutgers.edu}

\thanks{The first author is supported by the DFG funded project MA 2469/2-1.}
\thanks{The second author is supported in part by NSF grant DMS 1201474.} 

\begin{abstract}
The Heisenberg group $\mathbb{H}^1$ is known to be conformally equivalent to the CR sphere $\mathbb{S}^3$ minus a point. We use this fact, together with the knowledge of the tangential Cauchy-Riemann operator on the compact CR manifold $\mathbb{S}^3$, to solve the corresponding operator on $\mathbb{H}^1$.
\end{abstract}

\maketitle

\section{Introduction} \label{sect:intro}

Let $M$ be a strongly pseudoconvex CR manifold of real dimension 3. Then there is a distinguished complex 1-dimensional subspace $T^{1,0}$ of the complexified tangent space $\mathbb{C}TM$ of $M$, whose elements are called tangent vectors of type $(1,0)$. Let $\theta$ be a real contact 1-form on $M$, so that the kernel of $\theta$ is $T^{1,0} \oplus \overline{T^{1,0}}$. This induces a Hermitian inner product on $T^{1,0}$, given by $$(Z_1,Z_2) := d\theta(Z_1,i\Zbar_2), \qquad Z_1, Z_2 \in T^{1,0};$$ one can then define pointwise geometric quantities like the Webster scalar curvature $W$ on $M$. The pair $(M,\theta)$ is then called a pseudohermitian manifold; since then $\theta \wedge d\theta$ is a volume form on $M$, one can also define $L^r$ spaces of $(p,q)$ forms on $M$. If $\theta$ is a contact form satisfying the above conditions, then $v^2 \theta$ also satisfies the same conditions for any real-valued smooth function $v$ on $M$. The Hermitian inner product determined by $v^2 \theta$ is in general different from that determined by $\theta$, but they are conformally equivalent because for $Z_1, Z_2 \in T^{1,0}$, $$d(v^2 \theta)(Z_1, i\Zbar_2) = v^2 (d\theta)(Z_1, i\Zbar_2).$$

In this article, we look at two specific examples of pseudohermitian manifolds, that are conformally equivalent to each other. One is the Heisenberg group $\mathbb{H}^1$, which has zero Webster scalar curvature everywhere; another is the CR sphere $\mathbb{S}^3$ minus a point, which has constant Webster scalar curvature equal to 1. More precisely, the Heisenberg group $\mathbb{H}^1$ is the boundary of upper half space $\{\text{Im } z_2 > |z_1|^2 \}$ in $\mathbb{C}^2$, which we identify with $\mathbb{C} \times \mathbb{R}$ via $$\mathbb{C} \times \mathbb{R} \to \mathbb{H}^1$$ $$[z,t] \mapsto (z,t+i|z|^2).$$ It carries a standard contact form $$\theta = dt + i(z d\zbar - \zbar dz),$$ which gives $\mathbb{H}^1$ the structure of a pseudohermitian manifold. On the other hand, the CR sphere $\mathbb{S}^3$ is the boundary of unit ball $\{|\zeta| < 1 \}$ in $\mathbb{C}^2$, and it carries a standard contact form $$\htheta = i (\bar{\partial} - \partial) |\zeta|^2 = i \sum_{j=1}^2 (\zeta_j d\bar{\zeta}_j - \bar{\zeta}_j d\zeta_j).$$ If $p$ is the south pole $(0,-1)$ on $\mathbb{S}^2$, one can map $\mathbb{S}^3 \setminus \{p\}$ to $\mathbb{H}^1$ via stereographic projection: $$\zeta \in \mathbb{S}^3 \setminus \{p\} \mapsto (z,t+i|z|^2) \in \mathbb{H}^1$$ with $$z = \frac{\zeta_1}{1+\zeta_2}, \qquad t+i|z|^2 = i\frac{1-\zeta_2}{1+\zeta_2}.$$ We will always identify $\mathbb{H}^1$ with $\mathbb{S}^3 \setminus \{p\}$ this way, and pull $\theta$ and any function back from $\mathbb{H}^1$ to $\mathbb{S}^3 \setminus \{p\}$. Then $$\theta = G^2 \htheta,$$ where $$G(\zeta) := \frac{1}{|1+\zeta_2|},$$ so $(\mathbb{H}^1,\theta)$ and $(\mathbb{S}^3 \setminus \{p\},\htheta)$ are conformally equivalent to each other.

One can define inner products on functions and $(0,1)$ forms on $\mathbb{H}^1$ and $\mathbb{S}^3$. For functions $f,g$ on $\mathbb{H}^1$, we define their inner product by
$$\langle f, g \rangle_{\mathbb{H}^1} = \int_{\mathbb{H}^1} f \cdot \overline{g} \, \theta \wedge d\theta$$ whenever the integral makes sense. For $(0,1)$ forms $\alpha, \beta$ on $\mathbb{H}^1$, we define their inner product by 
$$\langle \alpha, \beta \rangle_{\mathbb{H}^1} = \int_{\mathbb{H}^1} (\alpha, \beta)_{\theta} \, \theta \wedge d\theta,$$ where $(\alpha,\alpha)_{\theta}$ is the Hermitian inner product on the dual bundle of $\overline{T^{1,0}}$ induced by $\theta$. Similarly, for functions $f,g$ on $\mathbb{S}^3$, we define their inner product by
$$\langle f, g \rangle_{\mathbb{S}^3} = \int_{\mathbb{S}^3} f \cdot \overline{g} \, \htheta \wedge d\htheta.$$ For $(0,1)$ forms $\alpha, \beta$ on $\mathbb{S}^3$, we define their inner product by 
$$\langle \alpha, \beta \rangle_{\mathbb{S}^3} = \int_{\mathbb{S}^3} (\alpha, \beta)_{\htheta} \, \htheta \wedge d\htheta.$$

One can also define $L^p$ spaces on both $\mathbb{H}^1$ and $\mathbb{S}^3$. For functions on $\mathbb{H}^1$, we define
$$\|f\|_{L^p(\mathbb{H}^1)} := \left( \int_{\mathbb{H}^1} |f|^p \theta \wedge d\theta \right)^{1/p}.$$
For $(0,1)$ forms on $\mathbb{H}^1$, we define
$$\|\alpha\|_{L^p_{(0,1)}(\mathbb{H}^1)} := \left( \int_{\mathbb{H}^1} |(\alpha,\alpha)_{\theta}|^{p/2} \theta \wedge d\theta \right)^{1/p}.$$
Similarly, one can define $L^p$ spaces of functions on $\mathbb{S}^3$: if $\hf$ is a function on $\mathbb{S}^3$,
$$\|\hf\|_{L^p(\mathbb{S}^3)} := \left( \int_{\mathbb{S}^3} |\hf|^p \htheta \wedge d\htheta \right)^{1/p},$$
as well as $L^p$ spaces of $(0,1)$ forms on $\mathbb{S}^3$: if $\hat{\alpha}$ is a (0,1) form on $\mathbb{S}^3$,
$$\|\hat{\alpha}\|_{L^p_{(0,1)}(\mathbb{S}^3)} := \left( \int_{\mathbb{S}^3} |(\hat{\alpha},\hat{\alpha})_{\htheta}|^{p/2} \htheta \wedge d\htheta \right)^{1/p}.$$ For future reference, we mention that if $f$ is a function on $\mathbb{H}^1$, then
$$\|f\|_{L^p(\mathbb{H}^1)} = \|G^{4/p} f\|_{L^p(\mathbb{S}^3)}.$$ 

Recall now the tangential Cauchy-Riemann operator on $\mathbb{H}^1$, which sends functions in $C^{\infty}_c(\mathbb{H}^1)$ to smooth $(0,1)$ forms on $\mathbb{H}^1$. If one takes the closure of the graph of this operator under the graph norm $L^4 \times L^2$, one obtains a densely defined closed linear operator $$\dbarb \colon L^4(\mathbb{H}^1) \to L^2_{(0,1)}(\mathbb{H}^1).$$ The formal adjoint of this operator sends $C^{\infty}_c$ $(0,1)$ forms on $\mathbb{H}^1$ to smooth functions on $\mathbb{H}^1$. The closure of this formal adjoint under the graph norm $L^4 \times L^2$ is written $$\dbarb^* \colon L^4_{(0,1)}(\mathbb{H}^1) \to L^2(\mathbb{H}^1).$$ 
Similarly, on $\mathbb{S}^3$, the tangential Cauchy-Riemann operator sends $C^{\infty}$ functions on $\mathbb{S}^3$ to smooth $(0,1)$ forms on $\mathbb{S}^3$. If one takes the closure of the graph of this operator under the graph norm $L^4 \times L^2$, one obtains a densely defined closed linear operator $$\hdbarb \colon L^4(\mathbb{S}^3) \to L^2_{(0,1)}(\mathbb{S}^3).$$ The formal adjoint of this operator sends $C^{\infty}$ $(0,1)$ forms on $\mathbb{S}^3$ to smooth functions on $\mathbb{S}^3$. The closure of this formal adjoint under the graph norm $L^4 \times L^2$ is written $$\hdbarb^* \colon L^4_{(0,1)}(\mathbb{S}^3) \to L^2(\mathbb{S}^3).$$ 

It is a well-known fact that on $(\mathbb{S}^3,\htheta)$, one can solve $\hat{\dbarb}$ and $\hat{\dbarb}^*$ with estimates; see e.g. Kohn-Rossi \cite{KR66} for the $L^2$ theory, and Greiner-Stein \cite[Proposition 10.9]{GS77}, Nagel-Stein \cite[Theorem 20]{NS79} for the $L^p$ theory. We will recall some of this in the next section\footnote{One can also refer to Folland-Stein \cite{FS74}, Rothschild-Stein \cite{RS76}, and Koenig \cite{Koe02} for the $L^p$ theory in higher dimensions. In fact, a lot is known even if $\mathbb{S}^3$ is replaced by the boundary of a weakly pseudoconvex domain of finite type in $\mathbb{C}^2$: see e.g. Kohn \cite{K86}, Boas-Shaw \cite{BS86}, and Christ \cite{C88}, \cite{C88b}.}. Our question is then the following: Can one solve $\dbarb$ (or $\dbarb^*$) on $\mathbb{H}^1$ using this knowledge on $\mathbb{S}^3$?

On the Heisenberg group, there are of course well-known integral formulas, with explicit kernels, that solve for us $\dbarb$ and $\dbarb^*$. On the other hand, these formula are very special, and works only because there is a group structure on the Heisenberg group. The method of solving $\dbarb$ and $\dbarb^*$ we describe below are more robust. This serves as a first step towards understanding the $\dbarb$ complex on some non-compact pseudohermitian CR manifolds of dimension 3.

More explicitly, let $\Zbar$ be the following vector field on $\mathbb{H}^1$: $$\Zbar = \frac{\partial}{\partial \zbar} - i z \frac{\partial}{\partial t},$$ and let $\omegabar = d\zbar$ be the dual $(0,1)$ form to $\Zbar$. Then if $u \in C^{\infty}_c(\mathbb{H}^1)$, we have $\dbarb u = (\Zbar u) \omegabar$. Thus solving $\dbarb \colon L^4(\mathbb{H}^1) \to L^2_{(0,1)}(\mathbb{H}^1)$ amounts to the following: one seeks conditions on a function $f \in L^2(\mathbb{H}^1)$, under which there exists a function $u \in L^4(\mathbb{H}^1)$, and a sequence $u_j \in C^{\infty}_c(\mathbb{H}^1)$, such that $$u_j \to u \text{ in $L^4(\mathbb{H}^1)$}, \quad \text{and} \quad \Zbar u_j \to f \text{ in $L^2(\mathbb{H}^1)$}.$$ When this holds, we say that $u \in L^4(\mathbb{H}^1)$ is a solution to the equation $$\Zbar u = f.$$ Similarly, let $\Zbar^*$ be the formal adjoint of $\Zbar$ under the inner product in $L^2(\mathbb{H}^1)$, i.e. $\Zbar^*$ is the differential operator satisfying $$\langle \Zbar f, g \rangle_{\mathbb{H}^1} = \langle f, \Zbar^* g\rangle_{\mathbb{H}^1}$$ for all $f, g \in C^{\infty}_c(\mathbb{H}^1)$. Then if $u \in C^{\infty}_c(\mathbb{H}^1)$, $\dbarb^* (u \omegabar) = \Zbar^* u$. Thus solving  $\dbarb^* \colon L^4(\mathbb{H}^1) \to L^2(\mathbb{H}^1)$ amounts to the following: one seeks conditions on a function $f \in L^2(\mathbb{H}^1)$, under which there exists a function $u \in L^4(\mathbb{H}^1)$, and a sequence $u_j \in C^{\infty}_c(\mathbb{H}^1)$, such that $$u_j \to u \text{ in $L^4(\mathbb{H}^1)$}, \quad \text{and} \quad \Zbar^* u_j \to f \text{ in $L^2(\mathbb{H}^1)$}.$$ When this holds, we say that $u \in L^4(\mathbb{H}^1)$ is a solution to the equation $$\Zbar^* u = f.$$

In order to state our results, we need the following extensions of $\Zbar$ and $\Zbar^*$ so that they become closed linear operators from $L^2(\mathbb{H}^1)$ to $L^{4/3}(\mathbb{H}^1)$. In fact, $\Zbar$ and $\Zbar^*$ are linear maps that preserve $C^{\infty}_c(\mathbb{H}^1)$. Thus we can take the closure of the graphs of these operators in the graph norm $L^2 \times L^{4/3}$, and obtain two closed linear operators $L^2(\mathbb{H}^1) \to L^{4/3}(\mathbb{H}^1)$. What we need are then the kernels of these closed linear operators: they are closed subspaces of $L^2(\mathbb{H}^1)$, and for convenience they will be referred to as \emph{the} kernel of $\Zbar$ and \emph{the} kernel of $\Zbar^*$. In other words, $u \in L^2(\mathbb{H}^1)$ is in the kernel of $\Zbar$, if and only if there exists a sequence $u_j \in C^{\infty}_c(\mathbb{H}^1)$, such that $$u_j \to u \text{ in $L^2(\mathbb{H}^1)$}, \quad \text{and} \quad \Zbar u_j \to 0 \text{ in $L^{4/3}(\mathbb{H}^1)$}.$$ Similarly, $u \in L^2(\mathbb{H}^1)$ is in the kernel of $\Zbar^*$, if and only if there exists a sequence $u_j \in C^{\infty}_c(\mathbb{H}^1)$, such that $$u_j \to u \text{ in $L^2(\mathbb{H}^1)$}, \quad \text{and} \quad \Zbar^* u_j \to 0 \text{ in $L^{4/3}(\mathbb{H}^1)$}.$$

Our main results are then the following:

\begin{thm} \label{thm:dbarb}
For any $f \in L^2(\mathbb{H}^1)$ that is orthogonal to the kernel of $\Zbar^*$ in $L^2(\mathbb{H}^1)$, there exists a solution $u \in L^4(\mathbb{H}^1)$ to the equation $\Zbar u = f$.
\end{thm}

\begin{thm} \label{thm:dbarb*}
For any $f \in L^2(\mathbb{H}^1)$ that is orthogonal to the kernel of $\Zbar$ in $L^2(\mathbb{H}^1)$, there exists a solution $u \in L^4(\mathbb{H}^1)$ to the equation $\Zbar^* u = f$.
\end{thm}

We will prove these theorems by reducing them to the corresponding statements for $\hdbarb \colon L^4(\mathbb{S}^3) \to L^2_{(0,1)}(\mathbb{S}^3)$ and $\hdbarb^* \colon L^4_{(0,1)}(\mathbb{S}^3) \to L^2(\mathbb{S}^3)$ on $\mathbb{S}^3$. The key will be two-fold:
\begin{enumerate}[(a)]
\item the conformal equivalence of $(\mathbb{H}^1,\theta)$ with $(\mathbb{S}^3 \setminus \{p\}, \htheta)$ mentioned above, where $\theta = G^2 \htheta$; and 
\item the fact that $G = |h|$ where $h = \frac{1}{1+\zeta_2}$ is a CR function on $\mathbb{S}^3 \setminus \{p\}$.
\end{enumerate}

We remark that on $\mathbb{H}^1$, it is easy to see that $\overline{Z}^* = - Z$. As a result, Theorems~\ref{thm:dbarb} and~\ref{thm:dbarb*} are equivalent to each other. But this is a feature specific to $\mathbb{H}^1$. In anticipation of a more general theory, we have therefore adopted a more robust approach below, that does not depend on this fact.

\section{The tangential Cauchy-Riemann complex on $\mathbb{S}^3$}

Before we proceed, let's first recall and clarify the definitions of the operators $\hdbarb \colon L^4(\mathbb{S}^3) \to L^2_{(0,1)}(\mathbb{S}^3)$ and $\hdbarb^* \colon L^4_{(0,1)}(\mathbb{S}^3) \to L^2(\mathbb{S}^3)$. Let $\hZbar$ be the vector field on $\mathbb{S}^3 \setminus \{p\}$ that satisfies 
\begin{equation} \label{eq:hZbar}
\hZbar = G \Zbar \quad \text{on $\mathbb{S}^3 \setminus \{p\}$,}
\end{equation} 
and $\homegabar$ be the $(0,1)$ form dual to $\hZbar$. Then for $u \in C^{\infty}_c(\mathbb{S}^3 \setminus \{p\})$, $$\hdbarb u = (\hZbar u) \homegabar.$$ Furthermore, we have the following lemma:

\begin{lemma} \label{lem:domainhdbarb}
For a function $u \in L^4(\mathbb{S}^3)$, the following are equivalent: 
\begin{enumerate}[(a)]
\item $u$ is in the domain of $\hdbarb \colon L^4(\mathbb{S}^3) \to L^2_{(0,1)}(\mathbb{S}^3)$;
\item there exists a sequence $v_j \in C^{\infty}(\mathbb{S}^3)$, so that $$v_j \to u \text{ in $L^4(\mathbb{S}^3)$}, \quad \text{and} \quad \hdbarb v_j \to \alpha \text{ in $L^2_{(0,1)}(\mathbb{S}^3)$}$$ for some $\alpha \in L^2_{(0,1)}(\mathbb{S}^3)$;
\item there exists a sequence $u_j \in C^{\infty}_c(\mathbb{S}^3 \setminus \{p\})$, so that $$u_j \to u \text{ in $L^4(\mathbb{S}^3)$}, \quad \text{and} \quad \hZbar u_j \to f \text{ in $L^2(\mathbb{S}^3)$}$$ for some $f \in L^2(\mathbb{S}^3)$.
\end{enumerate}
If all the above holds, then $$\hdbarb u = \alpha = f \homegabar \quad \text{in $L^2_{(0,1)}(\mathbb{S}^3)$}.$$
\end{lemma}
\begin{proof}
In fact, (a) and (b) are equivalent by definition, and it is clear that (c) implies (b). To see that (b) implies (c), note that if $v_j$ are as in (b), then one can take $u_j := (1-\chi_j) v_j$, where $\chi_j$ is a smooth function on $\mathbb{S}^3$ so that it is identically 1 in a (non-isotropic) ball of radius $\varepsilon_j$ centered at $p$, 0 outside a ball of radius $2 \varepsilon_j$, and $\|\hZbar \chi_j\|_{L^{\infty}} + \|\hat{Z} \chi_j\|_{L^{\infty}} \leq C \varepsilon_j^{-1}$. Here we choose $\varepsilon_j \to 0$ sufficiently rapidly so that $\|v_j\|_{L^4(\text{supp } \chi_j)} + \|\hZbar v_j\|_{L^2( \text{supp } \chi_j)} \to 0$ as $j \to \infty$. Then $u_j \in C^{\infty}_c (\mathbb{S}^3 \setminus \{p\})$, 
\begin{align*}
\|u_j - u\|_{L^4} 
&\leq \|v_j - u\|_{L^4} + \|\chi_j v_j\|_{L^4} \\
&\leq \|v_j - u\|_{L^4} + C \|v_j\|_{L^4( \text{supp } \chi_j)} \to 0 
\end{align*} 
as $j \to \infty$, and 
\begin{align*}
\|\hZbar u_j - f\|_{L^2} 
&\leq \|(1-\chi_j) \hZbar v_j - f\|_{L^2} + \|(\hZbar \chi_j)v_j\|_{L^2} \\
&\leq \|\hZbar v_j - f\|_{L^2} + \|\chi_j \hZbar v_j\|_{L^2} + \|(\hZbar \chi_j)\|_{L^4} \|v_j\|_{L^4(\text{supp } \chi_j)} \\
&\leq \|\hZbar v_j - f\|_{L^2} + C \|\hZbar v_j\|_{L^2 (\text{supp } \chi_j)} + C \|v_j\|_{L^4(\text{supp } \chi_j)} \to 0
\end{align*}
as $j \to \infty$. 

Now suppose (a) through (c) all holds. One can show that $\alpha = f \homegabar$ by testing them against a $(0,1)$ form that is smooth and compactly supported in $\mathbb{S}^3 \setminus \{p\}$. In fact, if $g \homegabar$ is such a form, then 
$$\langle \alpha, g\homegabar \rangle_{\mathbb{S}^3} = \langle u ,\hZbar^* g \rangle_{\mathbb{S}^3} = \langle f \homegabar, g\homegabar\rangle_{\mathbb{S}^3},$$ which shows that $\alpha = f \homegabar$ in $L^2_{(0,1)}(\mathbb{S}^3)$. They are equal to $\hdbarb u$ by definition.
\end{proof}

Next, let $\hZbar^*$ be the formal adjoint of $\hZbar$ under the inner product on $L^2(\mathbb{S}^3)$. In other words, $\hZbar^*$ is the differential operator satisfying
$$\langle \hZbar f, g \rangle_{\mathbb{S}^3} = \langle f, \hZbar^* g \rangle_{\mathbb{S}^3}$$ for all $f, g \in C^{\infty}_c(\mathbb{S}^3 \setminus \{p\})$. Then for $u \in C^{\infty}_c(\mathbb{S}^3 \setminus \{p\})$, $$\hdbarb^* (u \homegabar) = \hZbar^* u.$$ Furthermore, by the same argument as the one proving Lemma~\ref{lem:domainhdbarb}, we have:
\begin{lemma} \label{lem:domainhdbarb*}
For a $(0,1)$ form $\alpha \in L^4_{(0,1)}(\mathbb{S}^3)$, the following are equivalent: 
\begin{enumerate}[(a)]
\item $\alpha$ is in the domain of $\hdbarb^* \colon L^4_{(0,1)}(\mathbb{S}^3) \to L^2(\mathbb{S}^3)$;
\item there exists a sequence of smooth $(0,1)$ forms $\alpha_j$ on $\mathbb{S}^3$, so that $$\alpha_j \to \alpha \text{ in $L^4_{(0,1)}(\mathbb{S}^3)$}, \quad \text{and} \quad \hdbarb^* \alpha_j \to f \text{ in $L^2(\mathbb{S}^3)$}$$ for some $f \in L^2(\mathbb{S}^3)$;
\item there exists a sequence of functions $u_j \in C^{\infty}_c(\mathbb{S}^3 \setminus \{p\})$, so that $$u_j \homegabar \to \alpha \text{ in $L^4_{(0,1)}(\mathbb{S}^3)$}, \quad \text{and} \quad \hZbar^* u_j \to g \text{ in $L^2(\mathbb{S}^3)$}$$ for some $g \in L^2(\mathbb{S}^3)$.
\end{enumerate}
If all the above holds, then $$\hdbarb^* \alpha = f = g \quad \text{in $L^2(\mathbb{S}^3)$}.$$
\end{lemma}

We record for later use the following formula for $\hZbar^*$: if $u \in C^{\infty}_c(\mathbb{S}^3 \setminus \{p\})$, then $$\hZbar^* u = G^4 \Zbar^* (G^{-3} u).$$ This is because for any $v \in C^{\infty}_c(\mathbb{S}^3 \setminus \{p\})$, we have
$$\langle \hZbar^* u, v \rangle_{\mathbb{S}^3} = \langle u, \hZbar v \rangle_{\mathbb{S}^3} = \langle G^{-4} u, G \Zbar v \rangle_{\mathbb{H}^1} = \langle \Zbar^* G^{-3} u, v \rangle_{\mathbb{H}^1} = \langle G^4 \Zbar^* G^{-3} u, v \rangle_{\mathbb{S}^3}.$$ By a similar argument, for any $u \in C^{\infty}_c(\mathbb{S}^3 \setminus \{p\})$ and any integer $k$, we have
\begin{equation} \label{eq:hZbar*}
\hZbar^* u = \bar{h}^{-k} G^4 \Zbar^* (\bar{h}^k G^{-3} u):
\end{equation}
The key is that $h$ is a CR function on $\mathbb{H}^1$ (i.e. $\Zbar h = 0$ on $\mathbb{H}^1$). Thus  for any $v \in C^{\infty}_c(\mathbb{H}^1)$, we have
\begin{align*}
\langle \hZbar^* u, v \rangle_{\mathbb{S}^3} 
&= \langle G^{-4} u, G \Zbar v \rangle_{\mathbb{H}^1} \\ 
&= \langle \bar{h}^k G^{-3} u, \Zbar (h^{-k} v) \rangle_{\mathbb{H}^1} \\
&= \langle \Zbar^* (\bar{h}^k G^{-3} u), h^{-k} v \rangle_{\mathbb{H}^1} \\
&= \langle h^{-k} G^4 \Zbar^* (\bar{h}^k G^{-3} u), v \rangle_{\mathbb{S}^3}.
\end{align*}

To proceed further, an important fact we need is that the operators $\hdbarb \colon L^4(\mathbb{S}^3) \to L^2_{(0,1)}(\mathbb{S}^3)$ and $\hdbarb^* \colon L^4_{(0,1)}(\mathbb{S}^3) \to L^2(\mathbb{S}^3)$ have closed ranges. We will briefly recall the proof of this in what follows. To describe the range of these operators, we need to introduce another closure of the tangential Cauchy-Riemann operator, and another closure of its formal adjoint, which we define as follows:

First, the tangential Cauchy-Riemann operator sends $C^{\infty}$ functions on $\mathbb{S}^3$ to smooth $(0,1)$ forms on $\mathbb{S}^3$. We take the closure of this operator in the graph norm $L^2 \times L^2$, and obtain a closed linear operator
$$\hdbarb \colon L^2(\mathbb{S}^3) \to L^2_{(0,1)}(\mathbb{S}^3).$$ (Note this is a different closure than the one we took earlier!) In other words, we say that $u \in L^2(\mathbb{S}^3)$ is in the domain of $\hdbarb \colon L^2(\mathbb{S}^3) \to L^2_{(0,1)}(\mathbb{S}^3)$, if and only if there exists a sequence $u_j \in C^{\infty}(\mathbb{S}^3)$ such that $$u_j \to u \text{ in $L^2(\mathbb{S}^3)$}, \quad \text{and} \quad \hdbarb u_j \to \alpha \text{ in $L^2_{(0,1)}$}$$ for some $\alpha \in L^2_{(0,1)}(\mathbb{S}^3)$. In that case $\hdbarb u = \alpha$.

Next, we let $$\hdbarb^* \colon L^2_{(0,1)}(\mathbb{S}^3) \to L^2(\mathbb{S}^3)$$ be the Hilbert space adjoint of $\hdbarb \colon L^2(\mathbb{S}^3) \to L^2_{(0,1)}(\mathbb{S}^3)$. In other words, $\alpha \in L^2_{(0,1)}(\mathbb{S}^3)$ is in the domain of $\hdbarb^* \colon L^2_{(0,1)}(\mathbb{S}^3) \to L^2(\mathbb{S}^3)$, if and only if there exists a function $f \in L^2(\mathbb{S}^3)$ such that
$$\langle \alpha, \hdbarb u \rangle_{\mathbb{S}^3} = \langle f, u \rangle_{\mathbb{S}^3}$$ for all $u$ in the domain of $\hdbarb \colon L^2(\mathbb{S}^3) \to L^2_{(0,1)}(\mathbb{S}^3)$. In that case $\hdbarb^* \alpha = f$. Since the Hilbert space adjoint of a closed linear operator is again a closed linear operator, we note that $\hdbarb^* \colon L^2_{(0,1)}(\mathbb{S}^3) \to L^2(\mathbb{S}^3)$ is a closed linear operator as well.

Since $(\mathbb{S}^3,\htheta)$ is strongly pseudoconvex and embedded in $\mathbb{C}^2$, we have:
\begin{prop}
The ranges of the operators $$\hdbarb \colon L^2(\mathbb{S}^3) \to L^2_{(0,1)}(\mathbb{S}^3) \quad \text{and} \quad \hdbarb^* \colon L^2_{(0,1)}(\mathbb{S}^3) \to L^2(\mathbb{S}^3)$$ are closed subspaces of $L^2_{(0,1)}(\mathbb{S}^3)$ and $L^2(\mathbb{S}^3)$ respectively.
\end{prop} 

Let now $\mathcal{H}$ and $\mathcal{H}_1$ be the kernels of $\hdbarb \colon L^2(\mathbb{S}^3) \to L^2_{(0,1)}(\mathbb{S}^3)$ and $\hdbarb^* \colon L^2_{(0,1)}(\mathbb{S}^3) \to L^2(\mathbb{S}^3)$ respectively. These are closed subspaces of $L^2(\mathbb{S}^3)$ and $L^2_{(0,1)}(\mathbb{S}^3)$ respectively. Let $\hat{\Pi} \colon L^2(\mathbb{S}^3) \to \mathcal{H}$ and $\hat{\Pi}_1 \colon L^2_{(0,1)}(\mathbb{S}^3) \to \mathcal{H}_1$ be orthogonal projections onto these closed subspaces. They are continuous linear operators called the Szeg\"o projections. Now by the previous proposition, there exist continuous linear operators (called the relative solution operators) $$\hat{K}_1 \colon L^2_{(0,1)}(\mathbb{S}^3) \to \text{Dom}(\hdbarb \colon L^2(\mathbb{S}^3) \to L^2_{(0,1)}(\mathbb{S}^3)) \subseteq L^2(\mathbb{S}^3)$$ and $$\hat{K} \colon L^2(\mathbb{S}^3) \to \text{Dom}(\hdbarb^* \colon L^2_{(0,1)}(\mathbb{S}^3) \to L^2(\mathbb{S}^3)) \subseteq L^2_{(0,1)}(\mathbb{S}^3)$$ so that
$$\hdbarb \hat{K}_1 = I - \hat{\Pi}_1, \quad \text{and} \quad \hat{\Pi} \hat{K}_1 = 0 = \hat{K}_1 \hat{\Pi}_1 \quad \text{on $L^2_{(0,1)}(\mathbb{S}^3)$},$$ and $$\hdbarb^* \hat{K} = I - \hat{\Pi} \quad \text{and} \quad \hat{\Pi}_1 \hat{K} = 0 = \hat{K} \hat{\Pi} \quad \text{on $L^2(\mathbb{S}^3)$}.$$ 
Furthermore, it is known that $\hat{\Pi}$, $\hat{\Pi}_1$, $\hat{K}$ and $\hat{K}_1$ are pseudolocal operators: in particular, if $f$ is a smooth function on $\mathbb{S}^3$, then $\hat{\Pi} f$ and $\hat{K} f$ are smooth on $\mathbb{S}^3$; if $\alpha$ is a smooth $(0,1)$ form on $\mathbb{S}^3$, then $\hat{\Pi}_1 \alpha$ and $\hat{K}_1 \alpha$ are smooth on $\mathbb{S}^3$. 

Using this, we can prove:
\begin{lemma} \label{lem:domhdbarb*}
If $\alpha \in L^2_{(0,1)}(\mathbb{S}^3)$, then $\alpha$ is in the domain of $\hdbarb^* \colon L^2_{(0,1)}(\mathbb{S}^3) \to L^2(\mathbb{S}^3)$, if and only if there exists a sequence of smooth $(0,1)$ forms $\alpha_j$ on $\mathbb{S}^3$, such that $$\alpha_j \to \alpha \text{ in $L^2_{(0,1)}(\mathbb{S}^3)$}, \quad \text{and} \quad \hdbarb^* \alpha_j \to g \text{ in $L^2(\mathbb{S}^3)$}$$ for some $g \in L^2(\mathbb{S}^3)$. In that case, $\hdbarb^* \alpha = g$.
\end{lemma}
\begin{proof}
Suppose $\alpha$ is in the domain of $\hdbarb^* \colon L^2_{(0,1)}(\mathbb{S}^3) \to L^2(\mathbb{S}^3)$. Let $u = \hdbarb^* \alpha \in L^2(\mathbb{S}^3)$, and $u_j \in C^{\infty}(\mathbb{S}^3)$ be such that $u_j \to u$ in $L^2(\mathbb{S}^3)$. Let $\beta_j$ be a sequence of smooth $(0,1)$ forms such that $\beta_j \to \alpha$ in $L^2_{(0,1)}(\mathbb{S}^3)$. Then $$\alpha_j := \hat{K} u_j + \hat{\Pi}_1 \beta_j$$ are smooth $(0,1)$ forms, $$\lim_{L^2} \alpha_j = \hat{K} u + \hat{\Pi}_1 \alpha = (I - \hat{\Pi}_1) \alpha + \hat{\Pi}_1 \alpha = \alpha,$$ and $$\lim_{L^2} \hdbarb^* \alpha_j = \lim_{L^2} \hdbarb^* \hat{K} u_j = \lim_{L^2} (u_j - \hat{\Pi} u_j) = u - \hat{\Pi} u = \hdbarb^* \alpha.$$ This shows half of the implication. The reverse implication follows from the fact that $\hdbarb^* \colon L^2_{(0,1)}(\mathbb{S}^3) \to L^2(\mathbb{S}^3)$ is a closed linear operator, as we noted earlier.
\end{proof}

We now have a nice characterization of the kernels $\mathcal{H}$ and $\mathcal{H}_1$ of the operators $\hdbarb \colon L^2(\mathbb{S}^3) \to L^2_{(0,1)}(\mathbb{S}^3)$ and $\hdbarb^* \colon L^2_{(0,1)}(\mathbb{S}^3) \to L^2(\mathbb{S}^3)$ respectively:

\begin{lemma} \label{lem:kernelchar}
\begin{enumerate}[(i)]
\item $u \in L^2(\mathbb{S}^3)$ is in $\mathcal{H}$, if and only if there exists a sequence $v_j \in C^{\infty}(\mathbb{S}^3)$ such that $$v_j \to u \text{ in $L^2(\mathbb{S}^3)$}, \quad \text{and} \quad \hdbarb v_j \to 0 \text{ in $L^2_{(0,1)}(\mathbb{S}^3)$}.$$ 
\item $\alpha \in L^2(\mathbb{S}^3)$ is in $\mathcal{H}_1$, if and only if there exists a sequence of smooth $(0,1)$ forms $\alpha_j$ on $\mathbb{S}^3$ such that $$\alpha_j \to \alpha \text{ in $L^2_{(0,1)}(\mathbb{S}^3)$}, \quad \text{and} \quad \hdbarb^* \alpha_j \to 0 \text{ in $L^2(\mathbb{S}^3)$}.$$ 
\end{enumerate}
\end{lemma}

\begin{proof}
The proof of the first part of the lemma is immediate from the definition of $\hdbarb \colon L^2(\mathbb{S}^3) \to L^2_{(0,1)}(\mathbb{S}^3)$. The second part of the lemma follows from Lemma~\ref{lem:domhdbarb*}.
\end{proof}

We also have:
\begin{lemma} \label{lem:kerneldbarb}
Suppose $u \in L^2(\mathbb{S}^3)$ is in $\mathcal{H}$. Then there exists $u_j \in C^{\infty}_c(\mathbb{S}^3 \setminus \{p\})$ such that $$u_j \to u \text{ in $L^2(\mathbb{S}^3)$}, \quad \text{and} \quad \hZbar u_j \to 0 \text{ in $L^{4/3}(\mathbb{S}^3)$}.$$ It follows that $h^{-2} u \in L^2(\mathbb{H}^1)$ is in the kernel of $\Zbar$.
\end{lemma}

\begin{proof}
The first statement of the lemma follows by letting $u_j = (1-\chi_j) v_j$, where $v_j$ are as in part (i) of Lemma~\ref{lem:kernelchar}, and $\chi_j \in C^{\infty}(\mathbb{S}^3)$ are as in the proof of Lemma~\ref{lem:domainhdbarb}, except that now $\varepsilon_j$ are chosen such that $\|v_j\|_{L^2(\text{supp } \chi_j)} \to 0$ as $j \to \infty$. Note that 
\begin{align*}
\|u_j - u\|_{L^2} 
&\leq \|v_j - u\|_{L^2} + \|\chi_j v_j\|_{L^2} \\
&\leq \|v_j - u\|_{L^2} + C \|v_j\|_{L^2(\text{supp }\chi_j)} \to 0 \quad \text{as $j \to \infty$},
\end{align*}
and
\begin{align*}
\|\hZbar u_j\|_{L^{4/3}} 
&\leq \|(1-\chi_j) \hZbar v_j\|_{L^{4/3}} + \|(\hZbar \chi_j) v_j \|_{L^{4/3}} \\
&\leq C\|\hdbarb v_j\|_{L^2_{(0,1)}} + \|\hZbar \chi_j\|_{L^4} \|v_j\|_{L^2(\text{supp }\chi_j)} \\
&\leq C\|\hdbarb v_j\|_{L^2_{(0,1)}} + C \|v_j\|_{L^2(\text{supp }\chi_j)} \to 0
\quad \text{as $j \to \infty$}. 
\end{align*}

To see the second part of the lemma, note that if $u_j$ are as in the statement of the lemma, then $h^{-2} u_j \in C^{\infty}_c(\mathbb{H}^1)$, $$h^{-2} u_j \to h^{-2} u \text{ in $L^2(\mathbb{H}^1)$}, \quad \text{and} \quad \Zbar (h^{-2} u_j) = G^{-1} h^{-2} \hZbar u_j \to 0 \text{ in $L^{4/3}(\mathbb{H}^1)$}.$$ (We used (\ref{eq:hZbar}) in the equality in the previous line.) This is because 
$$\|h^{-2} u_j - h^{-2} u\|_{L^2(\mathbb{H}^1)} = \| u_j - u\|_{L^2(\mathbb{S}^3)} \quad \text{and} \quad \|G^{-1} h^{-2} \hZbar u_j \|_{L^{4/3}(\mathbb{H}^1)} = \|\hZbar u_j \|_{L^{4/3}(\mathbb{S}^3)}.$$ The claim then follows.
\end{proof}

Similarly, we have
\begin{lemma} \label{lem:kerneldbarb*}
Suppose $g \homegabar \in L^2(\mathbb{S}^3)$ is in $\mathcal{H}_1$. Then there exists $g_j \in C^{\infty}_c(\mathbb{S}^3 \setminus \{p\})$ such that $$g_j \to g \text{ in $L^2(\mathbb{S}^3)$}, \quad \text{and} \quad \hZbar^* g_j \to 0 \text{ in $L^{4/3}(\mathbb{S}^3)$}.$$ It follows that $\bar{h} G^{-3} g$ is in the kernel of $\Zbar^*$.
\end{lemma}

\begin{proof}
The proof of the lemma parallels that of Lemma~\ref{lem:kerneldbarb}.  One only needs to use (\ref{eq:hZbar*}) with $k = 1$ instead wherever we used (\ref{eq:hZbar}).
\end{proof}

We now come back to the operators $\hdbarb \colon L^4(\mathbb{S}^3) \to L^2_{(0,1)}(\mathbb{S}^3)$ and $\hdbarb^* \colon L^4_{(0,1)}(\mathbb{S}^3) \to L^2(\mathbb{S}^3)$ we studied at the beginning of this section, and discuss the solvability of these operators.

\begin{lemma} \label{lem:hZbar}
Suppose $\alpha \in L^2_{(0,1)}(\mathbb{S}^3)$ and is orthogonal to $\mathcal{H}_1$ in $L^2_{(0,1)}(\mathbb{S}^3)$. Then there is a function $u$ in the domain of $\hdbarb \colon L^4(\mathbb{S}^3) \to L^2_{(0,1)}(\mathbb{S}^3)$ such that $\hdbarb u = \alpha$. In particular, the range of $\hdbarb \colon L^4(\mathbb{S}^3) \to L^2_{(0,1)}(\mathbb{S}^3)$ is closed in $L^2_{(0,1)}(\mathbb{S}^3)$.
\end{lemma}
\begin{proof}
The key here is the $L^p$ theory of $\hat{K}_1$, which shows that $\hat{K}_1$ extends to a bounded linear operator from $L^2_{(0,1)}(\mathbb{S}^3) \to L^4(\mathbb{S}^3)$. Thus if $\alpha$ is as in the lemma, then $u:=\hat{K}_1 \alpha$ is in $L^4(\mathbb{S}^3)$. Furthermore, let $\alpha_j$ be a sequence of smooth $(0,1)$ forms such that $\alpha_j \to \alpha$ in $L^2_{(0,1)}(\mathbb{S}^3)$. Then $v_j := \hat{K}_1 \alpha_j \in C^{\infty}(\mathbb{S}^3)$, $v_j \to u$ in $L^4(\mathbb{S}^3)$, and $\hdbarb v_j = \alpha_j - \hat{\Pi}_1 \alpha_j \to \alpha - \hat{\Pi}_1 \alpha = \alpha$ in $L^2_{(0,1)}(\mathbb{S}^3)$. This completes the proof of the first statement in the lemma. It then follows that the range of $\hdbarb \colon L^4(\mathbb{S}^3) \to L^2_{(0,1)}(\mathbb{S}^3)$ is the orthogonal complement of $\mathcal{H}_1$ in $L^2_{(0,1)}(\mathbb{S}^3)$, which is a closed subspace of $L^2_{(0,1)}(\mathbb{S}^3)$.
\end{proof}

Similarly, using the $L^p$ theory of $\hat{K}$ instead, we have
\begin{lemma} \label{lem:hZbar*}
Suppose $f \in L^2(\mathbb{S}^3)$ and is orthogonal to $\mathcal{H}$ in $L^2(\mathbb{S}^3)$. Then there exists a $(0,1)$ form $\alpha$ in the domain of $\hdbarb^* \colon L^4_{(0,1)}(\mathbb{S}^3) \to L^2(\mathbb{S}^3)$ such that $\hdbarb^* \alpha = f$. In particular, the range of $\hdbarb^* \colon L^4_{(0,1)}(\mathbb{S}^3) \to L^2(\mathbb{S}^3)$ is closed in $L^2(\mathbb{S}^3)$.
\end{lemma}

We then have the following corollaries:
\begin{cor} \label{cor:hZbar}
Suppose $f \in L^2(\mathbb{S}^3)$, and $\langle f, g \rangle_{\mathbb{S}^3} = 0$ for all $g \in L^2(\mathbb{S}^3)$ with $g \homegabar \in \mathcal{H}_1$. Then there exists $\hu \in L^4(\mathbb{S}^3)$, and a sequence $\hu_j \in C^{\infty}_c(\mathbb{S}^3 \setminus \{p\})$, such that $\hu_j \to \hu$ in $L^4(\mathbb{S}^3)$, and $\hZbar \hu_j \to f$ in $L^2(\mathbb{S}^3)$.
\end{cor}
\begin{proof}
This follows from Lemma~\ref{lem:hZbar} and Lemma~\ref{lem:domainhdbarb}.
\end{proof}

Similarly, by Lemma~\ref{lem:hZbar*} and Lemma~\ref{lem:domainhdbarb*}, we have
\begin{cor} \label{cor:hZbar*}
Suppose $f \in L^2(\mathbb{S}^3)$, and $\langle f, g \rangle_{\mathbb{S}^3} = 0$ for all $g \in \mathcal{H}$. Then there exists $\hu \in L^4(\mathbb{S}^3)$, and a sequence $\hu_j \in C^{\infty}_c(\mathbb{S}^3 \setminus \{p\})$, such that $\hu_j \to \hu$ in $L^4(\mathbb{S}^3)$, and $\hZbar^* \hu_j \to f$ in $L^2(\mathbb{S}^3)$.
\end{cor}

\section{Proof of Theorem~\ref{thm:dbarb}}

We now proceed to prove Theorem~\ref{thm:dbarb}. Let $f$ be in $L^2(\mathbb{H}^1)$ and orthogonal to the kernel of $\Zbar^*$ in $L^2(\mathbb{H}^1)$. To solve $\Zbar u = f$, since $\hZbar = G \Zbar$, formally it suffices to solve 
\begin{equation} \label{eq:hZbar1}
\hZbar u = Gf.
\end{equation}
To do so one is tempted to use Corollary~\ref{cor:hZbar}. Unfortunately this does not work directly: while $Gf$ is in $L^2(\mathbb{S}^3)$, in general it is not orthogonal to all $g$ with $g\homegabar \in \mathcal{H}_1$. The key observation is the following: instead of solving (\ref{eq:hZbar1}), it suffices to solve
\begin{equation} \label{eq:hZbar2}
\hZbar (hu) = hGf,
\end{equation}
where $h$ is defined as at the end of Section 1. This works because $\hZbar h = 0$ on $\mathbb{S}^3 \setminus \{p\}$, as was observed earlier. Now if $f \in L^2(\mathbb{H}^1)$ and is orthogonal to the kernel of $\Zbar^*$ in $L^2(\mathbb{H}^1)$, then we claim the following:
\begin{enumerate}[(a)]
\item $hGf \in L^2(\mathbb{S}^3)$, and 
\item $hGf $ is orthogonal to all $g$ with $g \homegabar \in \mathcal{H}_1$. 
\end{enumerate}
In fact, 
$$\|hGf\|_{L^2(\mathbb{S}^3)} = \|f\|_{L^2(\mathbb{H}^1)},$$
proving claim (a).
To prove claim (b), we use Lemma~\ref{lem:kerneldbarb*}: if $g \homegabar$ is in $\mathcal{H}_1$, then $\bar{h} G^{-3} g$ is in the kernel of $\Zbar^*$. It follows that
\begin{align*}
\langle hGf , g\rangle_{\mathbb{S}^3} 
= \langle f, \bar{h} G^{-3} g \rangle_{\mathbb{H}^1}
= 0
\end{align*}
by our assumption on $f$. This proves claim (b) above.

Having now proved the claims (a) and (b) above, we invoke Corollary~\ref{cor:hZbar} with $hGf$ in place of $f$. Then we obtain some $\hu \in L^4(\mathbb{S}^3)$, and a sequence $\hu_j \in C^{\infty}_c(\mathbb{S}^3 \setminus \{p\})$, with $\hu_j \to \hu$ in $L^4(\mathbb{S}^3)$, and $\hZbar \hu_j \to hGf$ in $L^2(\mathbb{S}^3)$. Letting $$u := h^{-1} \hu \quad \text{and} \quad u_j:=h^{-1} \hu_j,$$ we have $u \in L^4(\mathbb{H}^1)$, $u_j \in C^{\infty}_c(\mathbb{H}^1)$, $u_j \to u$ in $L^4(\mathbb{H}^1)$, and $$\Zbar u_j = h^{-1} G^{-1} \hZbar \hu_j \to f$$ in $L^2(\mathbb{H}^1)$. (We used (\ref{eq:hZbar}) in the identity on the previous line.) Thus $u \in L^4(\mathbb{H}^1)$ is a solution to $\Zbar u = f.$ This proves our current Theorem.

\section{Proof of Theorem~\ref{thm:dbarb*}}

The proof of Theorem~\ref{thm:dbarb*} parallels that of Theorem~\ref{thm:dbarb}. Let $f \in L^2(\mathbb{H}^1)$ be orthogonal to the kernel of $\Zbar$ in $L^2(\mathbb{H}^1)$. Motivated by (\ref{eq:hZbar*}) with $k = 2$, one rewrites the equation $\Zbar^* u = f$ as 
\begin{equation} \label{eq:hZbar*2}
\hZbar^*(\bar{h}^{-2} G^3 u) = \bar{h}^{-2} G^4 f.
\end{equation}
(It does not work if we had used (\ref{eq:hZbar*}) with $k = 0$!) Now it is easy to check that $\bar{h}^{-2} G^4 f$ is in $L^2(\mathbb{S}^3)$, and orthogonal to $\mathcal{H}$ in $L^2(\mathbb{S}^3)$. In fact, $$\|\bar{h}^{-2} G^4 f\|_{L^2(\mathbb{S}^3)} = \|f\|_{L^2(\mathbb{H}^1)},$$ and if $H$ is in $\mathcal{H}$, then Lemma~\ref{lem:kerneldbarb*} shows that $h^{-2} H$ is in the kernel of $\Zbar$. Thus
\begin{align*}
\langle \bar{h}^{-2} G^4 f, H \rangle_{\mathbb{S}^3} 
&= \langle f, h^{-2} H \rangle_{\mathbb{H}^1} = 0
\end{align*}
by our assumption on $f$. Thus one can invoke Corollary~\ref{cor:hZbar*}, and obtain some $\hu \in L^4(\mathbb{S}^3)$, together with a sequence $\hu_j \in C^{\infty}_c(\mathbb{S}^3 \setminus \{p\})$, such that $$\hu_j \to \hu \text{ in $L^4(\mathbb{S}^3)$}, \quad \text{and} \quad \hZbar^* \hu_j \to \bar{h}^{-2} G^4 f \quad \text{ in $L^2(\mathbb{S}^3)$}.$$ Letting $$u: = \bar{h}^2 G^{-3} \hu \quad \text{and} \quad u_j := \bar{h}^2 G^{-3} \hu_j,$$ we have $u \in L^4(\mathbb{H}^1)$, $u_j \in C^{\infty}_c(\mathbb{H}^1)$, $u_j \to u$ in $L^4(\mathbb{H}^1)$, and $$\Zbar^* u_j = \bar{h}^2 G^{-4} \hZbar^* \hu_j \to f$$ in $L^2(\mathbb{H}^1)$ as desired. (We used (\ref{eq:hZbar*}) with $k = 2$ in the identity on the previous line.) Thus $u \in L^4(\mathbb{H}^1)$ is a solution to the equation $\Zbar^* u = f$. This completes the proof.

\end{document}